\documentclass[10pt]{amsart} 
\usepackage{amsthm}
\usepackage{amssymb}
\usepackage{amsfonts}
\usepackage{verbatim}
 \usepackage[colorlinks,backref]{hyperref}
\newcommand{\Arr}{\mathcal{A}}
\newcommand{\arr}{\mathcal{A}}

\newcommand{\CC}{\mathbb{C}}
\newcommand{\CP}{\mathbb{CP}}

\newcommand{\ZZ}{\mathbb{Z}}

\newcommand{\Ker}{\text{Ker}}
\newcommand{\arA}{\mathcal{A}}

\newcommand{\FF}{{\mathbb F}}
\newcommand{\Carr}{{\mathcal C}}

\theoremstyle{plain}

\newtheorem{thm}{Theorem}[section]

\newtheorem{cor}[thm]{Corollary}
\newtheorem{lem}[thm]{Lemma}

\theoremstyle{definition}

\newtheorem{remark}[thm]{Remark}

\newtheorem{example}[thm]{Example}

\usepackage{amsmath}
\usepackage{amscd}
\usepackage{tikz}
\usetikzlibrary{arrows,chains,matrix,positioning,scopes}
\makeatletter
\tikzset{join/.code=\tikzset{after node path={%
\ifx\tikzchainprevious\pgfutil@empty\else(\tikzchainprevious)%
edge[every join]#1(\tikzchaincurrent)\fi}}}
\makeatother

\tikzset{>=stealth',every on chain/.append style={join},
         every join/.style={->}}

\usepgflibrary{shapes.geometric}

\begin{document}

\title{Line arrangements and direct sums of free groups }
\author{Kristopher Williams}
\address{Department of Mathematics, University of Iowa, Iowa City, IA 52212, USA}
\email{kjwillia@math.uiowa.edu}
\urladdr{\href{http://math.uiowa.edu/~kjwillia}
{math.uiowa.edu/\char'176kjwillia}}

 \maketitle
 \begin{abstract}  We show that if the fundamental groups of the complements of two line arrangements in the complex projective plane are isomorphic to the same direct sum of free groups, then the complements of the arrangements  are homotopy equivalent. For any such arrangement $\Arr$, we also construct an arrangement $\Arr'$ such that $\Arr'$ is a complexified-real arrangement, the intersection lattices of the arrangements are isomorphic, and the complements of the arrangements are diffeomorphic.
 \end{abstract}

\section{Introduction}
\label{sec:intro}
%
%

Let $\Arr = \{H_0, \cdots, H_n\}$ be an arrangement of projective lines in $\CP^2$ with complement denoted by $M(\Arr ) = \CP^2 \setminus \cup_{i=0}^n H_i$.    The intersection lattice of the arrangement $L(\Arr)$ is the partially ordered set consisting of non-empty intersection of hyperplanes and is ordered by reverse inclusion (see \cite{OT-Arrs-MR1217488}).  Any information that may be determined from the intersection lattice is called combinatorial.

  One of the major questions in arrangements is to what extent the topology of $M(\Arr)$ is determined by the combinatorics of $\Arr$.    It is well known that the cohomology algebra of $M(\Arr)$ is so determined.   However, Rybnikov  has shown examples of two arrangements with isomorphic intersection lattices, but the fundamental groups of the complements of the arrangements are not isomorphic \cite{Ryb98}.

Examples of the latter type have proven difficult to find, with more results showing how the combinatorics may determine the topology.    One such result in this direction comes from  \cite{Fan-Directproducts-MR1460414}, where Fan introduced a graph  associated to an arrangement.    If the graph is a forest of trees, it is shown that the fundamental group of $M(\Arr)$ is isomorphic to a direct sum of free groups.    The converse was latter shown to hold by Eliyahu, Liberman, Schaps and Teicher \cite{ELST-Char-DSFG-converse-AG}.   More information about these results and the fundamental group of the complement of an arrangement are given in section \ref{sec:arr-DSFG}.

Using the work of Fan, we are able to show that arrangements with complements isomorphic to a direct sum of free groups have {\it nice} combinatorics  (see \cite{Jian-Yau-DiffeomorphicTypes-MR1283226} or section \ref{sec:homeomorphismtype} for definitions).     Using some constructions from matroid theory and properties of nice arrangements, in section \ref{sec:homeomorphismtype} we prove

\begin{thm}
Let $\Arr$ be an arrangement in $\CP^2$ such that the fundamental group of $M(\Arr)$ is isomorphic to a direct sum of free groups.   Then, there exists an arrangement $\Arr'$ with defining polynomial whose linear factors have only real coefficients, and $M(\Arr)$ is diffeomorphic to $M(\Arr')$.
\end{thm}

Let $D$ denote the complement of the variety defined by the algebraic plane curve $y^2-x^3-x^2=0$ in $\CC^2$.   One may show that $\pi_1(D) \cong \langle a,b : [a,b], a=b^{-1} \rangle \cong \ZZ$ and the 2-complex constructed from the presentation has the same homotopy type as $D$.   Therefore, $D$ is homotopy equivalent to $S^1 \vee S^2$.   For any arrangement $\Arr$ in $\CC^2$ consisting of one line $\pi_1(M(\Arr)) \cong \langle a: - \rangle \cong \ZZ$.  Further, $M(\Arr)$ is homotopy equivalent to $S^1$.  Thus there are examples of  complements of curves with isomorphic fundamental groups that are not homotopy equivalent.

 It is stated in \cite{Libgober-homotopyTypeComplementPlaneCurves-MR839126} that the following problem is still open:  construct two algebraic plane curves in $\CC^2$ such that the complements have isomorphic fundamental groups and the same Euler characteristic, but are not homotopy equivalent.  A related problem is to determine sufficient conditions on a family of curves so that the fundamental group determines the homotopy type of the complement.

In \cite{Falk-homotopy-types-MR1193601}, Falk gives examples of two arrangements in $\CC^3$ such that the complements have the same homotopy type, but the intersection lattices are not isomorphic.   As the complements are homotopy equivalent, they have isomorphic fundamental groups; in particular, the fundamental groups are isomorphic to $\ZZ^2 \oplus \FF_p \oplus \FF_q$ where $\FF_n$ is the free group on $n$ generators.    In section \ref{sec:homotopytype} we extend this result to

\begin{thm}\label{thm:homotopyThm}  Let $\Arr_1$ and $\Arr_2$ be arrangements in $\CP^2$ such that $\pi_1(M(\Arr_1)) \cong \pi_1(M(\Arr_2))$ and $\pi_1(M(\Arr_1))$ is isomorphic to a direct sum of free groups.  Then $M(\Arr_1)$ and $ M(\Arr_2)$ are homotopy equivalent.
\end{thm}

{\bf Acknowledgements} The author wishes to thank his advisor Richard Randell for many useful conversations.

\section{Arrangements and Direct Sums of Free Groups}
\label{sec:arr-DSFG}
%
%
\subsection{Fundamental Group of Arrangement Complements}
\label{subsec:fundgrp}

Let $\Arr $ be an arrangement of lines in $\CP^2$, and denote the complement of the arrangement by $M(\Arr) := \CP^2 \setminus \cup_{H \in \Arr} H$.    By choosing a line $H_0 \in \Arr $ to be the ``line at infinity,'' we will consider the arrangement $\Arr \setminus \{H_0\}$ as an arrangement in $\CC^2$.   As  $M(\Arr) := \CP^2 \setminus \cup_{H \in \Arr} H \cong \CC^2 \setminus \cup_{H \in \Arr, H \neq H_0} H =: M(\Arr \setminus H_0)$, we will study the complement of the projective arrangement by examining the complement of an affine arrangement.   

All presentations in this paper will be Randell-Arvola presentations.   A quick introduction to these presentations in the case of complexified-real arrangements may be found in \cite{Falk-homotopy-types-MR1193601}.    We work with these presentations in order to make use of the the following theorem.

\begin{thm}[\cite{CS-BraidMonodromy-MR1470093}] \label{thm:complex-homequiv} For any arrangement $\Arr$ in $\CC^2$, the standard CW-complex associated to the Randell-Arvola presentation of $\pi_1(M(\Arr))$ is homotopy equivalent to $M(\Arr)$.
\end{thm}

Recall that the standard CW-complex associated to a (finite) presentation $P$ has one 0-cell, a 1-cell for each generator (with both ends attached to the 0-cell), and a 2-cell for each relator (with boundary attached by following along the 1-cell associated to each generator with respect to orientation).   Given any two finite presentations of a group, it is possible to transform one presentation into the other via  a sequence of Tietze transformations.   However, one must exercise care when performing these transformations as they can alter the homotopy type of the associated CW-complex.    In\cite{Falk-homotopy-types-MR1193601}, Falk lists the following transformations as not affecting the homotopy type:
\begin{enumerate}
\item[(i)] Replace any relator $r$ by $w^{-1}r^{\pm1}w$ where $w$ is any word in the generators.
\item[(ii)]  Delete a generator $g$ and a relator $gw^{-1}$ where $w$ is a word in the generators that does not contain $g$, and in each relator replace $g$ with $w$.
\item[(iii)]  For any distinct relators $r$ and $s$, replace  $r$ with $rs$.
\end{enumerate}
Also listed is a transformation that changes the homotopy type by wedging the complex with a copy of $S^2$:
\begin{enumerate}
\item[(iv)]  Insert a relator that is a consequence of other relators.
\end{enumerate}
Any two presentations of a group may be attained by a sequence of these transformations and their inverses.

\subsection{Fan's Graph}
\label{subsec:FansGraph}
In \cite{Fan-Directproducts-MR1460414}, Fan defines a graph on an arrangement in $\CP^2$.   We will denote the graph of an arrangement by $F(\Arr)$.   The vertices of $F(\Arr)$ will consist of all points on the arrangement with multiplicity at least three.   For each projective line $H$ in the arrangement, let $\{a_{i,H}\}_{i=1}^{n_H}$ denote the collection of multiple points contained in $H$.   For any $a_{i,H}$ and $a_{i+1,H}$ draw an edge on  $H$ such that the endpoints of the edge are $a_{i,H}$ and $a_{i+1,H}$, no two edges intersect, and the edge does not intersect any other point on the arrangement of multiplicity at least two.    The result of repeating this construction over all lines in the arrangement is the graph $F(\Arr)$.   

 Fan proves that this graph is well-defined up to homotopy type and uses the graph to prove the following theorem.

\begin{thm}[\cite{Fan-Directproducts-MR1460414}, \cite{ELST-Char-DSFG-converse-AG}] \label{thm:Fan-DSFG} Let $\Arr$ be a line arrangement in $\CP^2$.    $F(\Arr)$ is a forest of trees if and only if $\pi_1 (M(\Arr)) $ is isomorphic to a direct sum of free groups.
\end{thm}

The forward direction was proven in \cite{Fan-Directproducts-MR1460414} and the backwards direction was proven in \cite{ELST-Char-DSFG-converse-AG}.

\subsection{Affine Nodal Arrangements}
\label{subsec:ANA}

In \cite{CDP-Analogs-Zariski-MR2153112}, Choudary, Dimca, and Papadima define an {\it affine nodal arrangement} as an arrangement of lines in $\CC^2$ such that the lines intersect only in double points.    Therefore, all higher order multiple points occur on the line at infinity.   If an affine nodal arrangement has $r \geq 2$ multiple points at infinity, then $\Arr$ is {\it split solvable} of type $\textbf{m}=(m_1,m_2,\dots, m_r)$ with multiple points of order $m_1 + 1, m_2 +1, \dots, m_r +1$.   A split solvable arrangement is an arrangement of $r$ sets of parallel lines, each set have $m_j$ lines for $1 \leq j \leq r$.   Any two lines not in the same set of parallel lines intersect in a double point.

If an arrangement has only one multiple point on the line at infinity, then the arrangement consists of set $m_1$ parallel lines.  We will abuse notation and say that the arrangement is affine nodal of type $(m_1)$.    An affine nodal arrangement of type ${\bf m} = (m_1, \dots, m_r)$ is a split solvable arrangement of type ${\bf m}$ if $r\geq 2$.    As the class of affine nodal arrangements has nice combinatorics, the combinatorics determine the diffeomorphism type of the complement (see \cite{Jian-Yau-DiffeomorphicTypes-MR1283226} or subsection \ref{subsec:nice}).  Therefore, we may determine a presentation for any affine nodal arrangement.

\begin{lem}
Any affine nodal arrangement  of type $(m_1, \dots, m_r)$ has Randell-Arvola presentation given by 
\[\langle a_{1,1},\dots,a_{1,m_1}, a_{2,1},\dots, a_{r,m_r} : [a_{i,j},a_{k,l}]    \rangle\]
where the relations are commutators and indexed by $1 \leq i < k \leq r, 1 \leq j \leq m_i, 1 \leq l \leq m_k$.
\end{lem}

Given any finite direct sum of free groups $G$ we may find an arrangement such that the fundamental group of the complement is isomorphic to $G$.

\begin{lem} Let $G = \bigoplus_{i=1}^n \FF_{m_i}$ for $m_i$ a positive integer for all $1 \leq i \leq n$, and let $\Arr$ be an affine nodal arrangement of type ${\bf m} = (m_1, \cdots, m_n)$.   Then $\pi_1(M(\Arr)) \cong G$.
\end{lem}
\section{Diffeomorphism Type}
\label{sec:homeomorphismtype}
%
%
The goal of this section is to prove 

\begin{thm}
Let $\Arr$ be an arrangement in $\CP^2$ such that the fundamental group of $M(\Arr)$ is isomorphic to a direct sum of free groups.   Then, there exists an arrangement $\Arr'$ with defining polynomial whose linear factors have only real coefficients, and $M(\Arr)$ is diffeomorphic to $M(\Arr')$.
\end{thm}

We begin by reviewing some necessary theorems and constructions from arrangements and matroids.  The constructions are explained in terms of the intersection lattice associated to the arrangement.   For more information about matroid theory see \cite{White-Matroids-MR849389}.

\subsection{Nice Arrangements}
\label{subsec:nice}

In  \cite{Jian-Yau-DiffeomorphicTypes-MR1283226}, Jiang and Yau define the class of {\it nice} arrangements in $\CP^2$ and prove the following theorem:
\begin{thm} \cite{Jian-Yau-DiffeomorphicTypes-MR1283226}
\label{thm:Nice-Homeo}
 If $\Arr $ is a nice arrangement and $\Arr '$ is another arrangement such that $L(\Arr) \cong L(\Arr' )$, then $M(\Arr )$ is diffeomorphic to $M(\Arr')$. 
\end{thm}

In order to define nice arrangements, Jiang and Yau construct a graph associated to the arrangement.  We recall their definition and terminology here for convenience.  We will denote this graph by $JY(\Arr)$ or simply by $JY$ if the arrangement is understood.    Let $VJY$ be the set of vertices of the graph and consist of all points of the arrangement with multiplicity at least three.    Let $EJY$ denote the set of edges of $JY$.   Any two vertices $v$ and $w$ that span a line in the arrangement will be associated to a unique edge denoted by $(v,w)$.

A {\it reduced path} of $JY$ is an $n$-tuple $(v_1, \dots, v_n)$ of vertices such that $(v_i,v_{i+1})$ is an edge in $EJY$ and $v_i, v_{i+1},v_{i+2}$ are not on the same line for $i=1, \dots, n-2$.   A {\it reduced circle} is a reduced path such that $v_1=v_n$ and the tuple is a reduced path upon any re-indexing.   

For any vertex $v_0 \in VJY$, the {\it star of $v_0$} is a subgraph of $JY$ denoted by $St(v_0)$ and consists of vertices $VSt(v_0) = \{v_0\} \cup \{ v \in VJY: (v,v_0) \in EJY \}$ and edges $ESt(v_0) = \{(v,w) \in EJY : v=v_0 \text{ or } w=v_0,$ or $ v,w $and $ v_0 $ span the same line in $ \Arr \}$.

An arrangement is called {\it nice} if there exists $v_1, \dots, v_n \in VJY$ such that $St(v_1)$, $\dots$, $St(v_n)$ are pairwise disjoint in $JY$ and $JY' = JY - \bigcup_{i=1}^n (ESt(v_i) \cup \{v_i\})$ contains no reduced circles.

Let $\Arr$ be an arrangement in $\CP^2$ such that the fundamental group of the complement is a direct sum of free groups.   By Theorem \ref{thm:Fan-DSFG}, we have that $F(\Arr)$ is a forest of trees.    One may see that $F(\Arr)$ is a subgraph of $JY(\Arr)$ and that adding the edges to form $JY(\Arr)$ will not introduce any reduced circles to the graph.   Therefore we have shown:

\begin{thm}
\label{thm:DSFG-nice} If $\Arr$ is an arrangement of lines in $\CP^2$ such that $\pi_1(M(\Arr))$ is isomorphic to a direct sum of free groups, then $\Arr$ is a nice arrangement.
\end{thm}

\begin{remark} By Theorem \ref{thm:DSFG-Homeo} we may conclude that any  two lattice isomorphic arrangements with fundamental groups isomorphic to a direct sum of free groups will have diffeomorphic complements.
\end{remark}

\subsection{Truncation} In the following sections we use terminology and constructions involving the intersection lattice of the arrangement.   For more information on these from a matroid theoretic point of view see \cite{White-Matroids-MR849389}.

Let $\Arr$ be a central, essential arrangement in $\CC^l$.     The {\it truncation} of the intersection lattice is an operation that removes all elements of the lattice of rank $l-1$ and lowers the rank of the top element to $l-1$.    Geometrically, the {\it truncation of an arrangement} is the arrangement formed by intersecting an arrangement $\Arr$ in $\CC^l$ with a generic hyperplane through the origin.   A generic hyperplane is one that preserves the intersection lattice from rank 0 to $l-2$.   The resulting arrangement is still central, and has intersection lattice isomorphic to the truncation of $L(\Arr)$.

The truncation of the lattice will be denoted by $T(L(\Arr))$.   If the repeated truncation of the matroid of an arrangement $\Arr $ yields a arrangement in $\CC^3$ such that the corresponding arrangement in $\CP^2$ has nice combinatorics, we may denote the truncation of the arrangement as $T_3(\Arr )$, and note that the arrangement is well-defined up to diffeomorphism type.   Note that if $\Arr$ is a complexified real arrangement, then $T_3(\Arr)$ is also a complexified real arrangement as we may choose the generic hyperplanes to preserve the real structure of the arrangement.

\subsection{Parallel Connection}

We briefly describe the construction of parallel connection.  For more information see \cite{Falk-Combi-and-OS-algs-MR1845492},  \cite{White-Matroids-MR849389}, \cite{EF-OS-Algebras-MR1719136}. 
      
A {\it base-pointed lattice} is a pair $(L(\Arr),H)$ where $H$ is a hyperplane in $\Arr$.   It will be useful in the next construction to think of the flats of the intersection lattice as the set of hyperplanes containing the subspace rather than as the intersection of the hyperplanes. 
 
 The {\it parallel connection} between base-pointed lattices $(L(\Arr_1), H_1)$ and $(L(\Arr_2), H_2)$ will be a base-pointed lattice 
 \begin{equation*}
(P,H^{'}):= (P((L(\Arr_1), H_1), (L(\Arr_2), H_2)), H')
 \end{equation*} with rank one elements $\{H : H \in \Arr_1 \setminus H_1, \text{ or } H \in \Arr_2 \setminus H_2\} \cup \{H'\}$.  Using the identification $H'=H_1=H_2$, we may take the flats of the lattice $P$ to be 
\begin{equation*}
\{ K : K \cap \Arr_1 \in L(\Arr_1), \text{ and } K \cap \Arr_2 \in L(\Arr_2) \}
\end{equation*}
The rank of a flat is given by $r_P (K) = r_1(K \cap \Arr_1) + r_2(K \cap \Arr_2) - r_1(K \cap \{H'\})$ where $r_i$ is the rank function associated to $L(\Arr_i)$.  

Given two arrangements, one may find an arrangement realizing the parallel connection as follows (\cite{EF-OS-Algebras-MR1719136}, \cite{White-Matroids-MR849389}).  Let $\arA_1$ and $\arA_2$ be central arrangements with defining polynomials $Q_1$ and $Q_2$, respectively.   By a change of coordinates, the hyperplane associated to the base-point is given by a coordinate hyperplane in each polynomial, i.e.

\begin{equation*}
Q_1(x) =x_1 \widehat{Q}_1(x_1,\cdots,x_n)   ~~~~ Q_2(y) = y_1 \widehat{Q}_2(y_1,\cdots,y_m)
\end{equation*}

In the parallel connection, the hyperplanes $y_1=0$ and $x_1=0$ will be identified.   Define a polynomial in coordinates $(x_1,\cdots,x_n,y_2,\cdots,y_m)$ by 
\begin{equation*}
Q= Q_1(x_1,\cdots,x_n)\widehat{Q}_2(x_1,y_2,\cdots,y_m).
\end{equation*}
$Q$ is a defining polynomial for an arrangement realizing the intersection lattice $P$ in $\CC^{n+m-1}$.  Denote the arrangement resulting from parallel connection by $P((\arA_1, H_1),(\arA_2, H_2))$.   

Different choices of base-point in the parallel connection may yield non-isomorphic lattices.   However, as a corollary to Corollary 4.3 in \cite{Falk-Proud-ParConn-BundlesOf-Arrs-MR1877716}, one may show

\begin{thm}
\label{thm:base-pointDiffeo}  Let $\arr_1$ and $\arr_2$ be central arrangements of hyperplanes.  Let $(L(\Arr_1),H_1)$ and $(L(\Arr_1),H_1')$ be base-pointed lattices with different base-points associated to the arrangement $\arA_1$, and let $(L(\Arr_2),H_2)$ be a base-pointed lattice associated to the arrangement $\arA_2$.   If $\Arr$ and $\Arr'$ are realizations of the lattices $P((L(\Arr_1), H_1), (L(\Arr_2), H_2))$ and $P((L(\Arr_1), H_1'), (L(\Arr_2), H_2))$ respectively, then $M(\Arr)$ and $M(\Arr')$ are diffeomorphic.
\end{thm}

If $\arr = P((\arA_1, H_1),(\arA_2, H_2))$ is an arrangement such that $T^3(\arr)$ is a matroid with nice combinatorics, then we may define the {\it 3-truncated parallel connection} of $\arA_1,\arA_2$ by $TP_3((\arA_1, H_1),(\arA_2, H_2))$, and see that it is defined up to diffeomorphism type of the complement.

The 3-truncated parallel connection has a simple geometric description if the initial arrangements are in $\CP^2$.   In this case one may simply consider the arrangement resulting from identifying the arrangements along the base-point hyperplane and  perturbing the rest of the arrangements into general position with respect to each other, being sure to maintain the respective intersection lattices.   Passing to the associated arrangement in $\CC^3$ realizes the 3-truncated parallel connection.   If the initial arrangements are complexified-real arrangements, this operation may be performed so that the resulting arrangement is complexified real as well.   One simply needs to perform the identification of the arrangements 'far enough away' from the intersection points of the respective arrangements.

\subsection{Direct Sums} 
Let $\Arr$ and $\Arr'$ be arrangements in $\CC^n$ and $\CC^m$ respectively.   The sum of the arrangements is defined as
\begin{equation*}
\Arr \oplus \Arr' = \{ H_i \oplus \CC^m, \CC^n \oplus H_j : H_i \in \Arr, H_j \in \Arr' \}
\end{equation*}
and is an arrangement in $\CC^{n+m}$ \cite{OT-Arrs-MR1217488} .  

If $T^3(\Arr \oplus \Arr')$ yields a matroid with nice combinatorics, we will call $D_3(\Arr, \Arr') := T_3(\Arr \oplus \Arr')$ the {\it 3-generic direct sum} of $\Arr $ and $\Arr'$.   The geometric description for two arrangements in $\CP^2$ is to place the arrangements in projective space in general position with respect to each other and consider the associated arrangement in $\CC^3$.   If the arrangements are complexified real arrangements, this operation may be performed so that the resulting arrangement is also complexified real arrangement.

\subsection{Construction}
In this section we will prove the following:    
\begin{thm}
\label{thm:DSFG-Homeo}
 Let $\Arr$ be an arrangement in $\CP^2$ such that $\pi_1(M(\Arr))$ is isomorphic to a direct sum of free groups. Then there exists an arrangement $\mathcal{B}$ that may be constructed by applying 3-truncated parallel connection and 3-generic direct sum to sequence of central, complexified-real arrangements in $\CC^2$, such that $M(\mathcal{B})$ is diffeomorphic to $M(\Arr)$ as arrangements in $\CP^2$.
\end{thm}

We divide the proof into a series of lemmas.  For each of the lemmas, we let $F$ denote Fan's graph of the arrangement.   We also let $\Arr$ denote both the arrangement in $\CP^2$ and $\CC^3$.  By Theorem \ref{thm:DSFG-nice} $\Arr$ has nice combinatorics, so we need only show that $\mathcal{B}$ and $\Arr$ are lattice isomorphic. 

\begin{lem}
If $F$ is an empty graph, then the arrangement $\mathcal{B}$ is constructed via a series of 3-generic direct sums.
\end{lem}
\begin{proof}
If the graph is empty, then $\Arr $ has no multiple points, i.e. all hyperplanes intersect in double points.   The lattice for the $\Arr $ is the rank three boolean lattice on $|\Arr |$ elements.   The arrangement may be realized by inductively applying 3-generic direct sums to the arrangement with defining polynomial $Q(\Arr)=x$ in $\CC^2$.  Denote the resulting arrangement by $\mathcal{B}$ and note that the lattices of $\Arr$ and $\mathcal{B}$ are isomorphic.
\end{proof}

\begin{lem}
\label{lem:tree-diffeo}
Suppose that $F$ is a tree, and every hyperplane has at least one multiple point.   Then $\mathcal{B} $ is constructible via a series of  3-truncated parallel connections.  
\end{lem}
\begin{proof}
 We proceed by induction on the number of multiple points, $k$, in the arrangement (i.e. vertices in $F(\Arr)$).
 
 If $k=1$, then, all hyperplanes meet at a single point.   Choosing homogeneous coordinates $[x:y:z]$, we may assume the arrangement is given by a defining polynomial $Q = zx(x-z)(x-2z)\cdots (x-(n-2)z)$ in $\CC^3$.

Let $\Arr' $  be the arrangement with defining polynomial $Q$ but considered as an arrangement in $\CC^2$.  Let $\Arr''$ be an arrangement in $\CC^2$ consisting of one hyperplane.   By forming the parallel connection along the hyperplanes given by $H'=Ker(z)$ and $H'' \in \Arr''$ we have
\begin{equation*}
\mathcal{B} = TP_3((\Arr' , H'), (\Arr'',H'')))
\end{equation*}
It is clear that $\mathcal{B}$ and $\Arr$ have isomorphic lattices.
 
 Now assume the theorem holds for $k \leq p$ and we will show it holds for $p+1$.    As the graph is a forest of trees, we may assume that there is a vertex $a$ such that $a$  has only one edge $e$ emanating from it.  The vertex $a$ has $m+1 \geq 3$ hyperplanes containing it, which we shall denote by $\{H_i \}_{i=0}^m$, letting $H_0$ denote the hyperplane corresponding to the edge $e$.   Consider the sub-arrangement $\Arr' = \Arr \setminus \{H_i\}_{i=1}^m$.   We may then associate to $\Arr'$ a graph $F'$ in the sense of Fan such that $F'$ is the graph $F$ with the vertex $a$ and edge $e$ removed.   
 
 The graph $F'$ has $p$ vertices and is a tree; therefore $\Arr'$ is constructible via 3-truncated parallel connection.
 
 Let $\Carr$ denote the arrangement in $\CC^2$ with defining polynomial $Q(\Carr) = x(x-y)(x-2y)\cdots(x-my)$.  Let $H_j^{\Carr} = \Ker(x - jy)$ We will now show that the intersection lattices of 
\begin{equation*}
\mathcal{B}=TP_3((\Carr , H_j^\Carr),(\Arr', H_0))
\end{equation*} and $\Arr $ are isomorphic.
 
We  first describe the elements of $L(\Arr)$.  The only rank 0 flat is the empty set, $L_0(\Arr)=\{\emptyset\}$.   The rank 1 flats are $\{H\}$ such that $H \in \Arr'$, and $\{H_i\}$ for $1 \leq i \leq m$.   The rank 2 flats are $B \in L_2(\Arr')$, $\{H_i\}_{i=0}^m$, and $\{H, H_i \}$ for all $H \in \Arr'$ and $1 \leq i \leq m$. (This follows from the fact that $a$ is the only multiple point in $\CP^2$ that the $H_i$ intersect in; therefore, each $H_i$ intersects any $H \in \Arr'$ in a double point).  Finally, the only rank 3 flat is the origin corresponding to the set of all hyperplanes.
 
 The flats of $\mathcal{B}$ can be determined from the flats of rank zero, one, and two arising in the parallel connection of $\Carr$ and $\Arr'$.    As the flats of $\Carr$ are $\emptyset$,  $\{H_0^{\Carr}\}, \dots \{H_m^{\Carr} \},$ and $\{H_0^{\Carr}, \dots, H_m^{\Carr}\}$, we have the flats of $T$ characterized by
\begin{align}
&\{K  : K \in  L(\Arr')  \} \\
&\{K \cup \{H_i^{'}\} : K \in  L(\Arr') , H_0 \notin K,  1 \leq i \leq m  \} \\
&\{K  \cup \{H_0^{'}, \dots, H_m^{'}\}: K \in  L(\Arr')  \}
\end{align}
such that the flat $K$ has rank zero, one, two.  Note that we are identifying $H_0 = H_0^\Carr$.

 Let $r_{\mathcal{B}}$,$r_{\Arr'}$, $r_\Carr$ be the respective rank functions of ${\mathcal{B}}, \Arr',$ and $\Carr$.    Then for any closed set $J$ in the parallel connection we have 
 
 \begin{equation*}
r_{\mathcal{B}}(J) := r_{\Arr'} (J \cap \Arr') + r_\Carr ( J \cap \Carr) - r_\Carr(J \cap \{H_0\})
\end{equation*}
 
 We now examine each case.   In case (1), we have that $r_{\Carr}(K \cap E_{\Carr}) - r_\Carr(K \cap {H_0}) = 0$.  If $H_0 \in K$, then $H_0^{\Carr} = H_0 \in K$. Therefore case (1) produces all elements of $L_i(\Arr')$ for $i=0,1,2$.
 
 For case (2) we have 
 
 \begin{align*}
r_{\mathcal{B}}(K \cup \{H_i^{\Carr}\}) & = r_{\Arr'} (K ) + r_\Carr ( \{H_i^{\Carr}\}) - r_\Carr(\emptyset) \\
 & = r_{\Arr'} (K )  + 1
\end{align*}

 Thus, $K$ may only be flats of rank zero or one in $\Arr'$, i.e. the empty set or consist of one hyperplane.
 
 Finally, in case (3) we have
 
 \begin{align*}
r_{\mathcal{B}}(K  \cup \{H_0^{\Carr}, \dots, H_m^{\Carr}\}) &= r_{\Arr'} (K \cup \{H_0\}) + r_\Carr (\{H_0^{\Carr}, \dots, H_m^{\Carr}\}) - r_\Carr(\{H_0^\Carr \}) \\
& = r_{\Arr'} (K \cup \{H_0\}) + 2 - 1 \\ &= r_{\Arr'} (K \cup \{H_0\})  + 1
\end{align*}
 
 Thus $K \cup \{H_0\}$ must have rank zero or one.   As the set is non-empty, it must have rank one and therefore $K = \emptyset$ or $K= \{H_0\}$.   
 
In summary, we have
 \begin{itemize}
 \item  (1) produces all flats of rank zero, one or two from $\Arr'$.
 \item (2) produces rank one flats of the form $\{H\}$ for $H \in \Arr'$ or $\{H_i^{\Carr}\}$ for $1 \leq i \leq m$.
 \item (3) produces  the rank two flat $\{H_0^{\Carr}, \dots, H_m^{\Carr}\}$.
 \end{itemize}
These are exactly the flats of $\Arr$ listed above with the same ranks, thus $L(\Arr)$ and $L(\mathcal{B})$ are  isomorphic and the lemma is proven.
 \end{proof}

\begin{lem}\label{lem:no-gp}Suppose that $F$ is a forest of trees, and every hyperplane has at least one multiple point.   Then $\mathcal{B} $ is constructible via a series of  3-truncated parallel connections and 3-generic direct sums. 
\end{lem}
\begin{proof}
Suppose that Fan's graph has multiple components and each hyperplane in the arrangement intersects one of these components non-trivially.   From the Lemma \ref{lem:tree-diffeo} each component is constructible via 3-truncated parallel connection.   Pairwise, the arrangements defining the components are in general position in $\CP^2$.   One can show that this is exactly the 3-generic direct sum of the arrangements corresponding to each component.   Letting this arrangement be denoted by $\mathcal{B}$, it is clear that $L(\mathcal{B})$ is isomorphic to $L(\Arr)$.
\end{proof}

\begin{proof}[Proof of Theorem \ref{thm:DSFG-Homeo}]
Let $\Arr'$ denote the sub-arrangement of $\Arr$ formed by removing all hyperplanes that only intersect the arrangement in double points, and let $m= |\Arr \setminus \Arr'|$.  By Lemma \ref{lem:no-gp}, one may build an arrangement $\mathcal{B}^{'}$ such that $L(\mathcal{B}_0^{'})$ and $L(\Arr')$ are isomorphic.   Let $\mathcal{C}$ denote the arrangement of one hyperplane in $\CC^1$ and let $\mathcal{B}_i^{'} = D_3(\mathcal{B}_{i-1}^{'}, \mathcal{C})$ for $i =1, \dots, m$.  Then setting $\mathcal{B} = \mathcal{B}_m^{'}$ one can see $L(\mathcal{B})$ is isomorphic to $L(\Arr)$ and the theorem is proven.\end{proof}

As each step in the construction may be performed in such a way as to preserve the real coefficients of the defining equations, we have the following corollary:

\begin{cor}
\label{cor:Real}
Let $\Arr$ be an arrangement in $\CP^2$ such that the fundamental group of $M(\Arr)$ is isomorphic to a direct sum of free groups.   Then, there exists an arrangement $\Arr'$ such that $\Arr'$ has a defining polynomial that factors into linear factors such that each factor has only real coefficients, and $M(\Arr)$ is diffeomorphic to $M(\Arr')$.
\end{cor}

\begin{remark} It is currently unknown if all arrangements with nice combinatorics have real representations with diffeomorphic complements.  Using techniques of geometric addition and geometric multiplication, one may construct an arrangement with a connected representation space, but no representation with only real coefficients. See \cite{White-CombGeoMR921064} for the techniques and use $x^2 + y^2 = -1$ as the desired initial variety for a representation space. \end{remark}

\begin{remark}
\label{clm:combed-graph}
In the next section we will want a particular representation for an arrangement.     Let $\Arr$ be an arrangement in $\CP^2$ such that $F(\Arr)$ is a tree.   Let any line be the line at infinity, and consider the induced arrangement in $\CC^2$ denoted by $d\Arr$.  Using Corollary \ref{cor:Real}, we may assume that $d\Arr$ is a complexified-real arrangement, thus may be depicted in the real plane given by standard $x$ and $y$ coordinates, with no vertical lines.   By abuse of notation, let $F(d\Arr)$ be the graph on this arrangement following the same rules given for Fan's graph in $\CP^2$.   By Theorem \ref{thm:DSFG-Homeo}, any modification to the position of the lines may be made as long as the intersection lattice is not altered.

We will show that a representation may be chosen such that for any multiple point of the arrangement that is not an endpoint or isolated point of $F(d\Arr)$ there exists a multiple point with larger $x$ value.    The claim easily holds for arrangements with only one multiple point.    

Suppose $d\Arr$ has $k+1$ multiple points and $m$ is non-isolated multiple point that is an endpoint of the graph.   Let $\{H_1, \dots, H_p\}$ be the lines intersecting at $m$ with $H_0$ containing more than one multiple point.   The sub-arrangement $\mathcal{B}^{'} = d\Arr \setminus \bigcup_{i=2}^p \{H_i\}$ has $k$ multiple points, and thus by induction there is an arrangement $\mathcal{B}$ with the same intersection lattice as $\mathcal{B}^{'}$, containing the corresponding line $H_1$, and none of the lines in $\mathcal{B}$ are vertical.  

Let $P$ be the set of points on the arrangement of multiplicity at least two.   One may choose a real number $N$ such that for all points in $N$, the $x$-coordinate of the point is less than $N$.    We now let $q$ be the point on the line $H_1$ with $x$-coordinate equal to $N+1$.   Choose lines $H_2, \dots, H_p$ such that each line has positive slope, and for all $H \in \mathcal{B}$, $H \cap H_i$ has multiplicity two and the multiple points on $H_i$ have $x$-coordinates greater than $N$.   

The resulting arrangement in $\CP^2$ has the same intersection lattice as $\Arr$, hence the complements of the arrangements are homeomorphic.\end{remark}

\begin{figure}[h!]
\begin{center} 
\begin{tabular}{cc}

		\begin{tikzpicture}[scale=0.3]
			 \draw (0,0) to (6,4);
			 	\draw (0,4)  to  (6,0);
			 	\draw (0,2)  to  (6,2);
			 	
			 	\draw (0,3) to (6,3);
			 	\draw (1,0)  to (5,6);
			 	\draw (1,6)  to (5,0);
				\end{tikzpicture}

	 \hspace{1cm} & \hspace{1cm}
	\begin{tikzpicture}[scale=0.3]
	
			\pgfsetcornersarced{\pgfpoint{5mm}{5mm}}
		
		 \pgfsetcornersarced{\pgfpointorigin}
		 \draw (-1.5,-1) to (7.5,5);
		\draw (-1.5,5)  to  (7.5,-1);
		
		\draw (-1.5,2)  to  (7.5,2);
		\draw (-1.5,3) to (7.5,3);
		
		\draw (4.1777777,-1)  to (6.5,6);
		\draw (6.83333,-1) to (4.5,6) ;

		\end{tikzpicture}
\end{tabular}
		\end{center}
				\caption{These arrangements have diffeomorphic complements and isomorphic lattices.  The preferred arrangement is on the right.}
											\label{fig:pref}
					\end{figure}
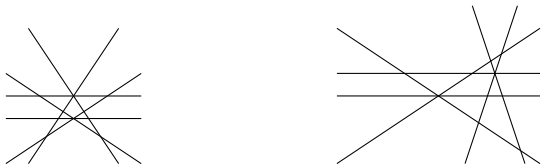

\section{Homotopy Type}
\label{sec:homotopytype}
%
%

\begin{lem}
\label{lem:DS-iso-he}
Any two affine nodal arrangements with isomorphic fundamental groups have diffeomorphic complements.
\end{lem}
\begin{proof}An affine nodal arrangement has nice combinatorics,   and  the intersection lattice  is determined by the fundamental group.    Therefore by Theorem \ref{thm:Nice-Homeo}, the complements of the arrangements are diffeomorphic.
\end{proof}

\begin{thm}\label{thm:affine-dsfg}  Let $\Arr$ be an arrangement in $\CP^2$ such that $\pi_1(M(\Arr))$ is isomorphic to a direct sum of free groups..  Then $M(\Arr)$ has the homotopy type of the complement of an affine nodal arrangement.
\end{thm}

\begin{proof}     We will show that the only transformations needed to change a Randell-Arvola presentation for $\pi_1(M(\Arr))$ into a Randell-Arvola presentation for an affine nodal arrangement are transformations of types $(i), (ii),$ and $(iii)$.   As these do not change the homotopy type of the associated 2-complex, this will show that $M(\Arr)$ has the homotopy type of the complement of an affine nodal arrangement.

By Theorem \ref{thm:DSFG-nice} and Corollary \ref{cor:Real}, we know that $\Arr$ has nice combinatorics and may assume that $\Arr$ is a complexified real arrangement.   Further, we choose the representation for $\Arr$ as described in Remark \ref{clm:combed-graph}.  Let $F=F(\Arr)$ be Fan's graph for the arrangement $\Arr$.   We will proceed by induction on the number of vertices in $F$.

Suppose $F$ has no vertices.   Then any two lines must intersect in only double points, and any two lines must intersect.   By Zariski's theorem, $\pi_1(M(\Arr))$ is isomorphic to the free abelian group on $n=|\Arr| -1$ generators.   Thus, $\Arr$ is an affine nodal arrangement of type $(m_1,\dots, m_n)$ for $m_i=1$ for all $i$.

Suppose $F$ has $k$ vertices and the theorem holds for all arrangements with less than $k$ vertices.    Let $d\Arr$ be a decone of the arrangement with respect to any line.    There are two possible cases to consider:  (1)the arrangement $d\Arr$ has no multiple points, or (2) $d\Arr$ has a multiple point that is an endpoint of an edge on the graph $F$ or an isolated point.   (Suppose not.   Then all multiple points are connected by two paths to vertices that lie on line at infinity.   As these two vertices both lie on the line at infinity, they are the same point or connected by an edge.   In either case, we have formed a cycle, contradicting graph being a forest.)

In case (1), as the decone has no multiple points, any two lines intersect in double points or not at all.   Therefore, $d\Arr$ is an affine nodal arrangement, and we are done.

In case (2), there are two possibilities:
\begin{itemize}
\item The graph has an isolated point.  Let $m$ denote the multiple point and let $H_1,\dots, H_j$ be the lines intersecting at $m$.
\item The graph has an edge with an endpoint. Let $m$ denote the multiple point and let $H_1,\dots, H_j$ be the lines intersecting at $m$ with $H_1$ denoting the line with more than one multiple point.   
\end{itemize}
Let $\mathcal{H} = \{H_1, \cdots, H_j\}$ and $\mathcal{B} = d\Arr \setminus \mathcal{H}$.  In either possibility, the lines $H_2, \dots, H_j$ intersect $\mathcal{B}$ in only double points. By Remark \ref{clm:combed-graph} we may assume that the multiple point $m$ and double points induced by the $H_i$ occur away from the rest of the arrangement.  


Therefore, after possibly performing some transformations of types $(i)$ and $(iii)$ we have a Randell-Arvola presentation given by 
\begin{equation*}
\langle h_1, \dots, h_j, b_1, \dots, b_k : [h_i, b_l], [h_1, h_2, \dots, h_j], R_\mathcal{B} \rangle 
\end{equation*}
where $2 \leq i \leq j, 1 \leq l \leq k$ and $R_\mathcal{B}$ are relators from the Randell-Arvola presentation for $\mathcal{B} \cup \{H_1\}$.   One should also note that any word in $R_\mathcal{B}$ is written in terms of $b_i$'s and $h_1$, and $h_i$ is a generator for a small loop around $H_i$ and $b_l$ is a generator for some line in $\mathcal{B}$.   Finally, we note that the bracket $[h_1, h_2, \dots, h_j]$ stands for the relations 
\begin{equation*}
h_j h_{j-1} \cdots h_1 = h_1 h_j \cdots h_2 = \cdots = h_{j-1} \cdots h_1 h_{j-1}
\end{equation*}

We apply a transformation of type $(ii)$ inverse by adding a generator $b_0$ and relator of the form $b_0( h_1 h_j \cdots h_2)^{-1}$, and replacing every occurrence of $ h_1 h_j \cdots h_2$ by $b_0$.   Next, apply a transformation of type $(ii)$ by deleting $h_1$ and the relator $h_1(b_0 h_2^{-1}h_3^{-1} \cdots h_{j-1}^{-1})^{-1}$.   We finish the transformation by replacing every occurrence of $h_1$ by $(b_0 h_2^{-1}h_3^{-1} \cdots h_{j-1}^{-1})$.  The resulting presentation is
\begin{equation*}
\langle h_2, \dots, h_j, b_0, b_1, \dots, b_k : [h_i, b_l], [b_0 h_2^{-1}h_3^{-1} \cdots h_{j-1}^{-1}, h_2, \dots, h_j], R_{\mathcal{B}, b_0, h_i} \rangle 
\end{equation*}

Applying transformations of types $(i)$ and $(iii)$ to $[b_0 h_2^{-1}h_3^{-1} \cdots h_{j-1}^{-1}, h_2, \dots, h_j]$ results in commutators $[b_0, h_i]$ for $2 \leq i \leq j$.   As all $h_i$'s commute with all $b_l$'s, all of the $h_i$'s may be removed from relators in $R_{\mathcal{B}, b_0, h_i}$ via transformation of types $(i)$ and $(iii)$.   The end result is a set of relators $R_{\mathcal{B}, b_0}$ that are identical to the relators in $R_{\mathcal{B}}$ except the letter $h_1$ has been changed to the letter $b_0$.   

Therefore, we have a presentation of the form
\begin{equation*}
P: = \langle h_2, \dots, h_j, b_0, b_1, \dots, b_k : [h_i, b_l], R_{\mathcal{B}, b_0} \rangle 
\end{equation*}
where $2 \leq i \leq j, 0 \leq l \leq k$ and $R_{\mathcal{B}, b_0}$ is a set of relators that do not involve any $h_i$.   Further, $\langle b_0, \dots, b_k : R_{\mathcal{B}, b_0} \rangle$ is the Randell-Arvola presentation for the arrangement $\mathcal{B} \cup \{H_1\}$.   When embedded into $\CP^2$ by adding the line at infinity, Fan's graph is a forest of trees with $k$ vertices.   Therefore, by the induction hypothesis there is sequence of transformations of types $(i)$, $(ii)$ and $(iii)$ that will change the presentation into a presentation of an affine nodal arrangement.    

Transformations of type $(i)$ and $(iii)$ will have no effect on the relators $[h_i, b_l]$, however, we must examine type $(ii)$.   Any move of this type will take the form of replacing a $b_l$ with a new generator $d$ written in terms of other $b_p$'s.   This transformation will result in a commutator of the form $[h_i, d(b_{n_1}\dotsb_{n_q})]$.   These may be transformed to $[h_i, d]$ by moves of type $(i)$ and $(iii)$ as the $h_i$ and $b_p$'s all commute.   The end result will simply be a change of letter for these commutators.

Thus, by following the sequence of transformation given by the subarrangement $\mathcal{B} \cup \{H_1\}$, the presentation $P$ may be transformed to a presentation of the same form as an affine nodal arrangement.   The transformations do not change the homotopy type of the complex, thus, $M(\Arr)$ has the homotopy type of an affine nodal arrangement.
\end{proof}

\begin{cor}(Theorem \ref{thm:homotopyThm} from the Introduction)  Let $\Arr_1$ and $\Arr_2$ be arrangements in $\CP^2$ such that $\pi_1(M(\Arr_1)) \cong \pi_1(M(\Arr_2)$ and $\pi_1(M(\Arr_1))$ is isomorphic to a direct sum of free groups.  Then $M(\Arr_1)$ and $M(\Arr_2)$ are homotopy equivalent.
\end{cor}

\begin{proof} 
Let $\Arr_3$ be an arrangement in $\CP^2$ such that $\Arr_3$ is an affine nodal arrangement and $\pi_1(M(\Arr_3)) \cong \pi_1(M(\Arr_1))$.    By Theorem \ref{thm:affine-dsfg}, $M(\Arr_1)$ is homotopy equivalent to $M(\Arr_3)$, and $M(\Arr_2)$ is homotopy equivalent to $M(\Arr_3)$.   Thus, we have the theorem.
\end{proof}

\section{Examples}
\label{sec:examples}
%
%

\begin{example} In \cite{Falk-homotopy-types-MR1193601}, two arrangements in $\CC^3$ were given with defining polynomials $Q(\Arr_1) = (x+y)(x-y)y(x+z)(x-z)z$ and $R(\Arr_2) = (x+z)(x-z)z(y+z)(y-z)(x-y-z)$.   Deconing with respect to the hyperplane defined by $z=0$ yields the affine arrangements depicted in Figure \ref{fig:exFalk}.

\begin{figure}[h!]
\begin{center}
	\begin{tikzpicture}[scale=0.3]

		\pgfsetcornersarced{\pgfpoint{5mm}{5mm}}
	
	 \pgfsetcornersarced{\pgfpointorigin}
	 \draw (0.2,5.8) to (5.8,0.2);
	\draw (0.2,0.2)  to  (5.8,5.8);
	\draw (0,3)  to  (6,3);
	\draw (2,0)  to (2,6);
	\draw (4,0)  to (4,6);
	
	\draw (3,-1) node [below] {$d\Arr_1$};
	
		 \draw (11,0) to (16,5);
		\draw (10,2)  to  (16,2);
		\draw (10,4) to (16,4);
		\draw (12,0)  to (12,6);
		\draw (14,0)  to (14,6);
		\draw (13,-1) node [below] {$d\Arr_2$};
		\end{tikzpicture}
		\caption{Arrangements from \cite{Falk-homotopy-types-MR1193601} that have homotopy equivalent complements.}
			\label{fig:exFalk}
\end{center}
			\end{figure}
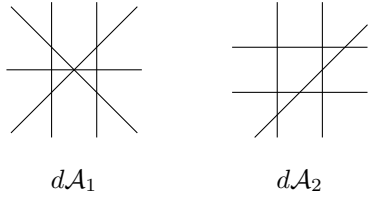

These arrangements are the first in an infinite family of pairs of arrangements in $\CC^3$ such that each pair has homotopy equivalent complements.   However, for each pair, the fundamental group of the complement is isomorphic to $\ZZ \oplus \ZZ \oplus \FF_p \oplus \FF_q$ where $\FF_n$ is the free group on $n$ generators.  In the example given, $\pi_1(M(\Arr_1)) \cong \ZZ \oplus \ZZ \oplus \FF_2 \oplus \FF_2$.

\end{example}

\begin{example} Consider the arrangements depicted in Figure \ref{fig:example_mine-CorrectForm}.   We will show explicitly the sequence of transformations between the Randell-Arvola presentations of the fundamental groups.

\begin{figure}[h!]
\begin{center} 
\begin{tabular}{cc}
	\begin{tikzpicture}[scale=0.3]

		\pgfsetcornersarced{\pgfpoint{5mm}{5mm}}
	
	 \pgfsetcornersarced{\pgfpointorigin}
	 \draw (-1.5,-1) to (7.5,5) node [right] {$c$};
	\draw (-1.5,5)  to  (7.5,-1) node [right] {$d$};
	
	\draw (-1.5,2)  to  (7.5,2) node [right] {$w$};
	\draw (-1.5,3) to (7.5,3) node [right] {$v$};
	
	\draw (4.1777777,-1)  to (6.5,6) node [right] {$a$};
	\draw (6.83333,-1) node [below] {$b$} to (4.5,6) ;
	
	\draw (3,-1) node [below] {$d\Arr_1$};
	\end{tikzpicture}
	 \hspace{1cm} & \hspace{1cm}
		\begin{tikzpicture}[scale=0.3]
		 \draw (1,0) to (6,5) node [right] {$d$}; 
		 \draw (0,1) to (5,6) node [right] {$c$};
		\draw (0,2)  to  (6,2)  node [right] {$f$};
		\draw (0,4) to (6,4) node [right] {$e$};
		\draw (2,0)  to (2,6)  node [above] {$a$};
		\draw (4,0)  to (4,6)  node [above] {$b$};
		\draw (3,-1) node [below] {$d\Arr_2$};
		\end{tikzpicture}
\end{tabular}
		\end{center}
				\caption{Arrangements that have homotopy equivalent complements.}
											\label{fig:example_mine-CorrectForm}
					\end{figure}
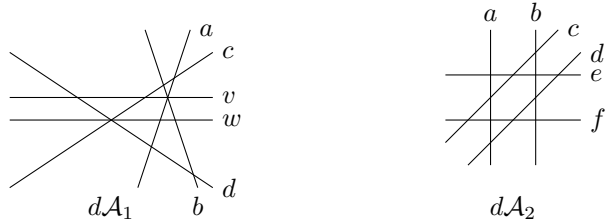
The presentations are given by 
\begin{align*}
\pi_1(M(d \arr_1)) \cong &  \langle a,b,c,d,e,f :& [a,c], [a,d], [a,e], [a,f], [b,c], [b,d], \\
 & & [b,e], [b,f], [c,e], [c,f], [d,e], [d,f] \rangle \\ 
\pi_1(M(d \arr_2)) \cong  & \langle a,b,c,d,v,w :& [b,d], [a,c], [b,w], [a,v,b], [b,c], \\
 & & [a,w], [c,v^b], [a,d], [c,w,d], [d,v^b] \rangle
\end{align*}
We first employ transformations of types (i) and (iii) to remove the conjugations from the commutators (i.e. replace $[c,v^b]$ by $[c,v]$ and replace $[d,v]$ by $[d,v]$ as $[b,c]$ and $[b,d]$ are relations).    

We next apply a transformation of type (ii) and type (ii) inverse by adding the generator $e$ and relation $e=vba$, then removing the generator $v$ by rewriting $v=eb^{-1}a^{-1}$.   Finally, apply a transformation of type (ii) and type (ii) inverse using the substitution $f=wxu$ and removing $w$ via the relation $w=fu^{-1}x^{-1}$.   This will result in the presentation given for $\pi_1(M(d \arr_1))$.    
\end{example}

\begin{example}
 Consider the arrangements defined by the polynomials 
 \begin{align*}
 Q(\Arr_1) &= xyz(y-2x)(x+z)(y+z)(2x-3y-6z)(4x-5y-10z) \\
 Q(\Arr_2) &= xyz(y-2x)(x+z)(y+z)(2x-3y-6z)(3x-4y-\frac{17}{2}z).
 \end{align*} In Figure \ref{fig:non-free} we have depicted the decones of the arrangements with respect to $z=0$.

\begin{figure}[h!]
\begin{center}
	\begin{tikzpicture}[scale=0.6]
		\pgfsetcornersarced{\pgfpoint{5mm}{5mm}}
	
	 \pgfsetcornersarced{\pgfpointorigin}
	    \draw (-2,-4) to (0.5,1);
		\draw (0,-4) to (0,1);
		\draw (-1,-4) to (-1,1);
		\draw (-3,0) to (4.5,0);
		\draw (-3,-1) to (4.5,-1);
		\draw (-3,-4) to (4.5,1);
			\draw (-2.5,-4) to (3.75,1);
	\draw (0,-5) node [below] {$d\Arr_1$};
	
	    \draw (8,-4) to (10.5,1);
	    		\draw (10,-4) to (10,1);
	    		\draw (9,-4) to (9,1);
	    		\draw (7,0) to (14.5,0);
	    		\draw (7,-1) to (14.5,-1);
	    		\draw (7,-4) to (14.5,1);
	    			\draw (7.5,-4) to (14.1777777,1);
	    	\draw (10,-5) node [below] {$d\Arr_2$};
		\end{tikzpicture}
		\end{center}
						\caption{Arrangements that have homotopy equivalent complements.   The fundamental group of the complement is not a direct sum of free groups.}
																	\label{fig:non-free}
						
		\end{figure}
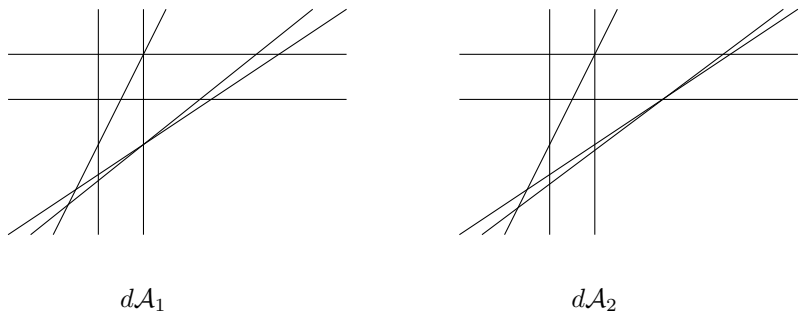

These arrangements arise as two dimensional generic sections of a pair of combinatorially distinct yet diffeomorphic arrangements in $\CC^4$ by techniques in \cite{EF-OS-Algebras-MR1719136}.   Therefore, by the Lefschetz theorem we know that the fundamental groups are isomorphic.

By using transformations of types $(i)$, $(ii)$, and $(iii)$ one may see that the Randell-Arvola presentations for the fundamental groups are equivalent, thus the complements of the arrangements are homotopy equivalent.   However, by Theorem \ref{thm:Fan-DSFG} the fundamental group is not a direct sum of free groups.  
					
\end{example}

\providecommand{\bysame}{\leavevmode\hbox to3em{\hrulefill}\thinspace}
\providecommand{\MR}{\relax\ifhmode\unskip\space\fi MR }
\providecommand{\MRhref}[2]{%
  \href{http://www.ams.org/mathscinet-getitem?mr=#1}{#2}
}
\providecommand{\href}[2]{#2}


\begin{thebibliography}{abcdef}

\bibitem[CDP05]{CDP-Analogs-Zariski-MR2153112}
A.~D.~R. Choudary, A.~Dimca, and {\c{S}}.~Papadima, \emph{Some analogs of
  {Z}ariski's theorem on nodal line arrangements}, Algebr. Geom. Topol.
  \textbf{5} (2005), 691--711 (electronic). \MR{MR2153112 (2006f:32038)}

\bibitem[CS97]{CS-BraidMonodromy-MR1470093}
D.~C. Cohen and A.~I. Suciu, \emph{The braid monodromy of plane algebraic
  curves and hyperplane arrangements}, Comment. Math. Helv. \textbf{72} (1997),
  no.~2, 285--315. \MR{MR1470093 (98f:52012)}

\bibitem[EF99]{EF-OS-Algebras-MR1719136}
C.~J. Eschenbrenner and M.~J. Falk, \emph{Orlik-{S}olomon algebras and {T}utte
  polynomials}, J. Algebraic Combin. \textbf{10} (1999), no.~2, 189--199.
  \MR{MR1719136 (2000k:05072)}

\bibitem[ELST10]{ELST-Char-DSFG-converse-AG}
M.~Eliyahu, E.~Liberman, M.~Schaps, and M.~Teicher, \emph{The characterization
  of a line arrangement whose fundamental group of the complement is a direct
  sum of free groups}, Algebraic \& Geometric Topology \textbf{10} (2010),
  no.~3, 1285--1304.

\bibitem[Fal93]{Falk-homotopy-types-MR1193601}
M.~Falk, \emph{Homotopy types of line arrangements}, Invent. Math. \textbf{111}
  (1993), no.~1, 139--150. \MR{MR1193601 (93j:52020)}

\bibitem[Fal01]{Falk-Combi-and-OS-algs-MR1845492}
\bysame, \emph{Combinatorial and algebraic structure in {O}rlik-{S}olomon
  algebras}, European J. Combin. \textbf{22} (2001), no.~5, 687--698,
  Combinatorial geometries (Luminy, 1999). \MR{MR1845492 (2002f:52023)}

\bibitem[Fan97]{Fan-Directproducts-MR1460414}
K.-M. Fan, \emph{Direct product of free groups as the fundamental group of the
  complement of a union of lines}, Michigan Math. J. \textbf{44} (1997), no.~2,
  283--291. \MR{MR1460414 (98j:14039)}

\bibitem[FP02]{Falk-Proud-ParConn-BundlesOf-Arrs-MR1877716}
M.~J. Falk and N.~J. Proudfoot, \emph{Parallel connections and bundles of
  arrangements}, Topology Appl. \textbf{118} (2002), no.~1-2, 65--83,
  Arrangements in Boston: a Conference on Hyperplane Arrangements (1999).
  \MR{MR1877716 (2002k:52033)}

\bibitem[JY94]{Jian-Yau-DiffeomorphicTypes-MR1283226}
T.~Jiang and S.~S.-T. Yau, \emph{Diffeomorphic types of the complements of
  arrangements of hyperplanes}, Compositio Math. \textbf{92} (1994), no.~2,
  133--155. \MR{MR1283226 (95e:32042)}

\bibitem[Lib86]{Libgober-homotopyTypeComplementPlaneCurves-MR839126}
A.~Libgober, \emph{On the homotopy type of the complement to plane algebraic
  curves}, J. Reine Angew. Math. \textbf{367} (1986), 103--114. \MR{839126
  (87j:14044)}

\bibitem[OT92]{OT-Arrs-MR1217488}
P.~Orlik and H.~Terao, \emph{Arrangements of hyperplanes}, Grundlehren der
  Mathematischen Wissenschaften [Fundamental Principles of Mathematical
  Sciences], vol. 300, Springer-Verlag, Berlin, 1992. \MR{MR1217488
  (94e:52014)}

\bibitem[Ryb93]{Ryb98}
G.~Rybnikov, \emph{On the fundamental group of the complement of a complex
  hyperplane arrangement}, \url{http://arxiv.org/abs/math/9805056}, 1993.

\bibitem[Whi86]{White-Matroids-MR849389}
N.~White (ed.), \emph{Theory of matroids}, Encyclopedia of Mathematics and its
  Applications, vol.~26, Cambridge University Press, Cambridge, 1986.
  \MR{MR849389 (87k:05054)}

\bibitem[Whi87]{White-CombGeoMR921064}
N.~White (ed.), \emph{Combinatorial geometries}, Encyclopedia of Mathematics
  and its Applications, vol.~29, Cambridge University Press, Cambridge, 1987.
  \MR{921064 (88g:05048)}

\end{thebibliography}
\end{document}